\documentclass{amsart}

\usepackage{a4wide,amssymb,color}%,refcheck}
\usepackage[all]{xy}

\newcommand{\IN}{\mathbb N}
\newcommand{\IZ}{\mathbb Z}

\newcommand{\XX}{\mathcal X}

\newtheorem{theorem}{Theorem}

\newtheorem{lemma}[theorem]{Lemma}

\theoremstyle{definition}

\title{The Golomb space is topologically rigid}
\author{Taras Banakh, Dario Spirito, S\l awomir Turek}
\address{T.Banakh: Ivan Franko National University of Lviv (Ukraine) and Jan Kochanowski University in Kielce (Poland)}
\email{t.o.banakh@gmail.com}
\address{D.Spirito: Dipartimento di Matematica e Fisica, Universit\`a degli Studi ``Roma Tre", Roma (Italy)}
\email{spirito@mat.uniroma3.it}
\address{S.Turek: Cardinal Stefan Wyszy\'nski University in Warsaw (Poland)}
\email{s.turek@uksw.edu.pl}

\begin{document}
\begin{abstract} The {\em Golomb space} $\IN_\tau$ is the set $\IN$ of positive integers endowed with the topology $\tau$ generated by the base consisting of arithmetic progressions $\{a+bn:n\ge 0\}$ with coprime $a,b$. We prove that the Golomb space $\IN_\tau$ is topologically rigid in the sense that its homeomorphism group is trivial. This resolves a problem posed by the first author at {\tt Mathoverflow} in 2017.
\end{abstract}

\maketitle

\section{Introduction}

In the AMS Meeting announcement \cite{Brown} M.~Brown introduced an amusing topology $\tau$ on the set $\IN$ of positive integers turning it into a connected Hausdorff space. The topology $\tau$ is generated by the base consisting of arithmetic progressions $a+b\IN_0:=\{a+bn:n\in\IN_0\}$ with coprime parameters $a,b\in\IN$. Here by $\IN_0=\{0\}\cup\IN$ we denote the set of non-negative integer numbers. 

In \cite{SS} the topology $\tau$ is called the {\em relatively prime integer topology}. This topology was popularized by Solomon Golomb \cite{Golomb59}, \cite{Golomb61} who observed that the classical Dirichlet theorem on primes in arithmetic progressions is equivalent to the density of the set $\Pi$ of prime numbers in the topological space $(\IN,\tau)$. As a by-product of such popularization efforts, the topological space $\IN_\tau:=(\IN,\tau)$ is known in General Topology as the {\em Golomb space}, see \cite{Szcz}, \cite{Szcz13}. 

The topological structure of the Golomb space $\IN_\tau$ was studied by Banakh, Mioduszewski and Turek \cite{BMT} who proved that the space $\IN_\tau$ is not topologically homogeneous (by showing that $1$ is a fixed point of any homeomorphism of $\IN$). Motivated by this results, the authors of \cite{BMT} posed a problem of the topological rigidity of the Golomb space. This problem was also repeated by the first author at {\tt Mathoverflow} \cite{MO}. A topological space $X$ is defined to be {\em topologically rigid} if its homeomorphism group is trivial. 

The main result of this note is the following theorem answering the above problem.

\begin{theorem}\label{t:main} The Golomb space $\IN_\tau$ is topologically rigid.
\end{theorem}

The proof of this theorem will be presented in Section~\ref{s:main} after some preparatory work made in Sections~\ref{s:2}, \ref{s:3}. The idea of the proof belongs to the second author who studied in \cite{DS} the rigidity properties of the Golomb topology on a Dedekind ring with removed zero, and established in \cite[Theorem 6.7]{DS} that the homeomorphism group of the Golomb topology on $\IZ\setminus\{0\}$ consists of two homeomorphisms. The proof of Theorem~\ref{t:main} is a modified (and simplified) version of the proof of Theorem 6.7 given in \cite{DS}. It should be mentioned that the Golomb topology on Dedekind rings with removed zero was  studied by Knopfmacher,  Porubsk\'y \cite{KP}, Clark, Lebowitz-Lockard, Pollack \cite{CLP}, and  Spirito \cite{DS}, \cite{DS2}.

%\begin{theorem} The set $\Pi$ of prime numbers is a dense metrizable subspace of the Golomb space $\IN_\tau$.
%\end{theorem}

%For any number $x\in\IN$ let $\Pi_x$ be the set of prime divisors of $x$.

%\begin{theorem}\label{mainH} Any homeomorphism $h:\IN_\tau\to\IN_\tau$ of the Golomb space has the following properties:
%\begin{enumerate}
%\item $h(1)=1$;
%\item $h(\Pi)=\Pi$;
%\item $\Pi_{h(x)}=h(\Pi_x)$ for every $x\in \IN$.
%\end{enumerate}
%\end{theorem}

%Theorem~\ref{mainH} implies that the Golomb space is not topologically homogeneous. This answer a problem \cite{Ban}, posed by the first author on Mathoverflow.

%\begin{problem} Is the Golomb space rigid?
%\end{problem}

%We recall that a topological space $X$ is {\em rigid} if its homeomorphism group is trivial.

\section{Preliminaries and notations}

In this section we fix some notation and recall some known results on the Golomb topology.  For a subset $A$ of a topological space $X$ by $\overline{A}$ we denote the closure of $A$ in $X$. 

A {\em poset} is a set $X$  endowed with a partial order $\le$. A subset $L$ of a partially oredered set $(X,\le)$ is called
\begin{itemize}
\item {\em linearly ordered} (or else  a {\em chain}) if any points $x,y\in L$ are {\em comparable} in the sense that $x\le y$ or $y\le x$;
\item an {\em antichain} if any two distinct elements $x,y\in A$ are not comparable.
\end{itemize} 

%We denote by $\IN_0$ the set of non-negative integer numbers and by $\IN$ the set of positive integers (so that $\IN_0=\{0\}\cup\IN$); $\IZ$ stands for the ring of integer numbers.
%For a point $x\in\IN$ by $\tau_x=\{U\in\tau\colon x\in U\}$ we denote the family of open neighborhoods of $x$ in the topology $\tau$ of the Golomb space $\IN_\tau$.

%For two numbers $x,y\in\IN$ by $\gcd(x,y)$ their greatest common divisor is denoted. 
By $\Pi$ we denote the set of prime numbers. For a number $x\in\IN$  we denote by $\Pi_x$ the set of all prime divisors of $x$. Two numbers $x,y\in\IN$ are {\em coprime} iff $\Pi_x\cap\Pi_y=\emptyset$. For a number $x\in\IN$ let $x^\IN:=\{x^n:n\in\IN\}$ be the set of all  powers of $x$.

%The following lemma is well-known in Number Theory as the Chinese Remained Theorem \cite{??}.

%\begin{lemma}\cite{Chinese} For any coprime numbers $a,b\in\IN$ and any $x,y\in \IN$ the intersection $(x+a\IN)\cap (y+b\IN)$ is not empty.
%\end{lemma}

%\begin{theorem}[Chinese Remainder Theorem]\label{Chinese} For any numbers $a_1,\dots,a_n\in\IZ$ and $b_1,\dots,b_n\in\IN$ the following conditions are equivalent:
%\begin{enumerate}
%\item the intersection $\bigcap_{i=1}^n(a_i+b_i\IN)$ is not empty;
%\item the intersection $\bigcap_{i=1}^n(a_i+b_i\IN)$ contains an infinite arithmetic progression;
%\item the intersection $\bigcap_{i=1}^n(a_i+b_i\IN)$ is infinite;
%\item for any $i,j$ the number $a_i-a_j$ is divisible by $\gcd(b_i,b_j)$.
%\end{enumerate}
%\end{theorem}

%The following lemma is a classical result of Dirichlet \cite[S.VI]{Dirichlet}, see also \cite[Ch.7]{Ap}.

%\begin{lemma}[Dirichlet]\label{l:Dirichlet} Each arithmetic progression $a+b\IN_0$ with $\gcd(a,b)=1$ contains a prime number. 
%\end{lemma}

%A number $q\in\IN$ is called {\em square-free} if it is not divided by the square $p^2$ of any prime number $p$.

For a number $x\in\IN$ and a prime number $p$ let $l_p(x)$ be the largest integer number such that $p^{l_p(x)}$ divides $x$. The function $l_p(x)$ plays the role of logarithm with base $p$.

The following formula for the closures of basic open sets in the Golomb topology was established in \cite[2.2]{BMT}.

\begin{lemma}[Banakh, Mioduszewski, Turek]\label{basic} For any $a,b\in\IN$ 
$$\overline{a+b\IN_0}=\IN\cap\bigcap_{p\in\Pi_b}\big(p\IN\cup (a+p^{l_p(b)}\IZ)\big).$$
\end{lemma}

Also we shall heavily exploit the following lemma, proved in \cite[5.1]{BMT}.

\begin{lemma}[Banakh, Mioduszewski, Turek]\label{l:BMT} Each homeomorphism $h:\IN_\tau\to\IN_\tau$ of the Golomb space has the following properties:
\begin{enumerate}
\item $h(1)=1$;
\item $h(\Pi)=\Pi$;
\item $\Pi_{h(x)}=h(\Pi_x)$ for every $x\in\IN$;
\item $h(x^\IN)=h(x)^{\IN}$ for every $x\in\IN$.
\end{enumerate}
\end{lemma}

 %A number $x$ is square-free if and only if $l_p(x)\le 1$ for any prime number $p$.

%A function $f:X\to Y$ is called {\em finite-to-one} if for each $y\in Y$ the preimage $f^{-1}(y)$ is finite.

%A family $\F$ of subsets of a set $X$ is called a {\em filter} if
%\begin{itemize}
%\item $\emptyset\notin\F$;
%\item for any $A,B\in\F$ their intersection $A\cap B\in\F$;
%\item for any sets $F\subset E\subset X$ the inclusion $F\in\F$ implies $E\in\F$.
%\end{itemize}

%In the subsequent proofs we shall exploit the  following two known results of Number Theory. The first one is a general version of the Chinese Remainder Theorem, which can be found in \cite[3.12]{J}.

%\begin{theorem}[Chinese Remainder Theorem]\label{Chinese} For any numbers $a_1,\dots,a_n\in\IZ$ and $b_1,\dots,b_n\in\IN$ the following conditions are equivalent:
%\begin{enumerate}
%\item the intersection $\bigcap_{i=1}^n(a_i+b_i\IN)$ is not empty;
%\item the intersection $\bigcap_{i=1}^n(a_i+b_i\IN)$ contains an infinite arithmetic progression;
%\item the intersection $\bigcap_{i=1}^n(a_i+b_i\IN)$ is infinite;
%\item for any $i,j$ the number $a_i-a_j$ is divisible by $\gcd(b_i,b_j)$.
%\end{enumerate}
%\end{theorem}

Let $p$ be a prime number and $k\in\IN$. Let $\IZ$ be the ring of integer numbers, $\IZ_{p^k}$ be the residue ring $\IZ/p^k\IZ$, and $\IZ_{p^k}^\times$ be the multiplicative group of invertible elements of the ring $\IZ_{p^k}$. It is well-known that $|\IZ_{p^k}^\times|=\phi(p^k)=p^{k-1}(p-1)$.
The structure of the group $\IZ_{p^k}^\times$ was described by Gauss in  \cite[art.52--56]{Gauss} (see also Theorems 2 and 2' in Chapter 4 of \cite{NT}).

\begin{lemma}[Gauss]\label{l:Gauss} Let $p$ be a prime number and $k\in\IN$. 
\begin{enumerate}
\item If $p$ is odd, then the group $\IZ^\times_{p^k}$ is cyclic;
\item If $p=2$ and $k\ge 2$, then the element $-1+2^k\IZ$ generates a two-element cyclic group $C_2$ in $\IZ^\times_{2^k}$, the element $5+2^{k}\IN$ generates a cyclic subgroup $C_{2^{k-2}}$ of order $2^{k-2}$ in $\IZ^\times_{2^k}$ such that $\IZ^\times_{2^k}=C_2\oplus C_{2^{k-2}}$.
\end{enumerate}
\end{lemma}

\section{Golomb topology versus the $p$-adic topologies on $\IN$}\label{s:2}

Let $p$ be any prime number. Let us recall that the {\em $p$-adic} topology on $\IZ$ is generated by the base consisting of the sets $x+p^n\IZ$, where $x\in\IZ$ and $n\in\IN$. This topology induces the {\em $p$-adic} topology on the subset $\IN$ of $\IZ$. It is generated by the base consisting of the sets $x+p^n\IN_0$ where $x,n\in\IN$. The following lemma is a special case of Proposition 3.1 in \cite{DS}. 

\begin{lemma}\label{l:DS} For any clopen subset $\Omega$ of $\IN_\tau\setminus p\IN$, and any $x\in\Omega$, there exists $n\in\IN$ such that $x+p^n\IN_0\subset \Omega$.
\end{lemma}

\begin{proof} Since the set $p\IN$ is closed in $\IN_\tau$, the set $\Omega$ is open in $\IN_\tau$ and hence $x+p^kb\IN_0\subset \Omega$ for some $k\in\IN$ and $b\in\IN$, which is coprime with $px$. We claim that $x+p^n\IN_0\subset \Omega$. To derive a contradiction, assume that $x+p^n\IN_0\setminus\Omega$ contains some number $y$. Since $\Omega$ is closed in $\IN_\tau\setminus p\IN$, there exist $m\ge n$ and $d\in\IN$ such that $d$ is coprime with $p$ and $y$, and $(y+p^md\IN_0)\cap\Omega=\emptyset$. It follows that $y+p^m\IN_0\subset (x+p^n\IN_0)+p^m\IN_0\subset x+p^n\IN_0$. Since $p\notin\Pi_b\cup\Pi_d$, we can apply the Chinese Remainder Theorem \cite[3.12]{J} and conclude that $\emptyset\ne (y+p^m\IN)\cap\bigcap_{q\in\Pi_b\cup\Pi_d}q\IN$.
Applying Lemma~\ref{basic} and taking into account that the set $\Omega$ is clopen in $\IN_\tau\setminus p\IN$, we conclude that
\begin{multline*}
\emptyset\ne (y+p^m\IN_0)\cap\Big(\bigcap_{q\in\Pi_b\cup\Pi_d}q\IN\Big)= (x+p^n\IN_0)\cap \Big(\bigcap_{q\in\Pi_b}q\IN\Big)\cap(y+p^m\IN_0)\cap\Big(\bigcap_{q\in\Pi_d}q\IN\Big)\subseteq\\
\overline{x+p^nb\IN_0}\cap\overline{y+p^nd\IN_0}\subset \overline{\Omega}\cap\overline{(\IN\setminus p\IN)\setminus\Omega)}\subset p\IN,
\end{multline*}
which is not possible as the sets $x+p^n\IN_0$ and $p\IN$ are disjoint. This contradiction shows that $x+p^n\IN_0\subset\Omega$.
\end{proof}

A subset of a topological space is {\em clopen} if it is closed and open. By the {\em zero-dimensional reflection} of a topological space $X$ we understand the space $X$ endowed with the topology generated by the base consisting of clopen subsets of the space $X$.

\begin{lemma}\label{l:refl} The $p$-adic topology on $\IN\setminus p\IN$ coincides with the zero-dimensional reflection of the subspace $\IN_\tau\setminus p\IN$ of the Golomb space $\IN_\tau$.
\end{lemma}

\begin{proof} Lemma~\ref{l:DS} implies that the $p$-adic topology $\tau_p$ on $\IN\setminus p\IN$ is stronger than the topology $\zeta$ of zero-dimensional reflection on $\IN_\tau\setminus p\IN$. To see that the $\tau_p$ coincides with $\zeta$, it suffices to show that for every $x\in \IN\setminus p\IN$ and $n\in\IN$ the basic open set $\IN\cap (x+p^n\IZ)$ in the $p$-adic topology is clopen in the subspace topology of $\IN_\tau\setminus p\IN\subset\IN_\tau$. By the definition, the set $\IN\cap (x+p^n\IZ)$ is open in the Golomb topology. To see that it is closed in $\IN_\tau\setminus p\IN$, take any point $y\in (\IN\setminus p\IN)\setminus (x+p^n\IZ)$ and observe that the Golomb-open neighborhood $y+p^n\IN_0$ of $y$ is disjoint with the set $\IN\cap(x+p^n\IZ)$.
\end{proof} 

For every prime number $p$, consider the countable family $$\XX_p=\big\{\overline{a^\IN}:a\in \IN\setminus p\IN,\;a\ne 1\big\},$$
where the closure $\overline{a^\IN}$ is taken in the $p$-adic topology on $\IN\setminus p\IN$, which coincides with the topology of zero-dimensional reflection of the Golomb topology on $\IN\setminus p\IN$ according to Lemma~\ref{l:refl}. 

The family $\XX_p$ is endowed with the partial order $\le$ defined by $X\le Y$ iff $Y\subseteq X$. So, $\XX_p$ is a poset carrying the partial order of reverse inclusion.

\begin{lemma}\label{l:isomor} For any prime number $p$, any homeomorphism $h$ of the Golomb space $\IN_\tau$ induces an order isomorphism $$h:\XX_p\to\XX_{h(p)},\;\;h:\overline{a^\IN}\mapsto h(\overline{a^\IN})=\overline{h(a)^\IN}$$
of the posets $\XX_p$ and $\XX_{h(p)}$. 
\end{lemma}

\begin{proof} By Lemma~\ref{l:BMT}, $h(1)=1$ and $h(p)$ is a prime number. First we show that $h(p\IN)=h(p)\IN$. Indeed, for any $x\in p\IN$ we have $p\in\Pi_x$ and by Lemma~\ref{l:BMT}, $h(p)\in h(\Pi_x)=\Pi_{h(x)}$ and hence $h(x)\in h(p)\IN$ and $h(p\IN)\subset h(p)\IN$. Applying the same argument to the homeomorphism $h^{-1}$, we obtain $h^{-1}(h(p)\IN)\subset p\IN$, which implies the desired equality $h(p\IN)=h(p)\IN$. The bijectivity of $h$ ensures that $h$ maps homeomorphically the space $\IN_\tau\setminus p\IN$ onto the space $\IN_\tau\setminus h(p)\IN$.

Then $h$ also is a homeomorphism of the spaces $\IN\setminus p\IN$ and $\IN_\tau\setminus h(p)\IN$ endowed with the zero-dimensional reflections of their subspace topologies inherited from the Golomb topology of $\IN_\tau$. By Lemma~\ref{l:refl}, these reflection topologies on $\IN\setminus p\IN$ and $\IN\setminus h(p)\IN$ coincide with the $p$-adic  and $h(p)$-adic topologies on $\IN\setminus p\IN$ and $\IN\setminus h(p)\IN$, respectively.

By Lemma~\ref{l:BMT}, for any $a\in\IN\setminus (\{1\}\cup p\IN)$ we have $$h(a)^\IN=h(a^\IN)\subseteq h(\IN\setminus  p\IN)=\IN\setminus h(p)\IN$$ and by the continuity of $h$ in the topologies of zero-dimensional reflections, we get $h(\overline{a^\IN})=\overline{h(a^\IN)}=\overline{h(a)^\IN}$. The same argument applies to the homeomorphism $h^{-1}$. This implies that $$h:\XX_p\to \XX_{h(p)},\;\;h:\overline{a^\IN}\mapsto h(\overline{a^\IN})=\overline{h(a)^\IN},$$
is a well-defined bijection. It is clear that this bijection preserves the inclusion order and hence it is an order isomorphism between the posets $\XX_p$ and $\XX_{h(p)}$. 
\end{proof}

\section{The order structure of the posets $\XX_p$}\label{s:3}

In this section, given a prime number $p$, we investigate the order-theoretic structure of the poset $\XX_p$. 

For every $n\in\IN$ denote by $\pi_n:\IN\to\IZ_{p^n}$ the homomorphism assigning to each number $x\in \IN$ the residue class $x+p^n\IZ$. Also for $n\le m$ let $$\pi_{m,n}:\IZ_{p^m}\to \IZ_{p^n}$$ be the ring homomorphism assigning to each residue class $x+p^m\IZ$ the residue class $x+p^n\IZ$. It is easy to see that $\pi_n=\pi_{m,n}\circ\pi_m$. Observe that the multiplicative group $\IZ_{p^n}^\times$ of invertible elements of the ring $\IZ_{p^n}$ coincides with the set $\IZ_{p^n}\setminus p\IZ_{p^n}$ and hence has cardinality $p^n-p^{n-1}=p^{n-1}(p-1)$.

First we establish the structure of the elements $\overline{a^\IN}$ of the family $\XX_p$.

%For number $m\ge n$, consider the quotient homomorphism $\pi_{m,n}:\IZ_{p^m}\to\IZ_{p^n}$, $\pi_{m,n}:x+p^m\IZ\mapsto x+p^n\IZ$.

\begin{lemma}\label{l:S5} If  for some $a\in \IN\setminus p\IZ$ and $n\in\IN$ the element $\pi_n(a)$ has order $\ge \max\{p,3\}$ in the multiplicative group $\IZ_{p^n}^\times$, then $\overline{a^\IN}=\pi_n^{-1}(\pi_n(a)^\IN)$.%for any $m\ge n$ the group $\pi_{m,n}^{-1}(\pi_n(a)^\IN)$ is generated by the element $\pi_m(a)$.
\end{lemma}

\begin{proof}  Let $B=b^\IN$ be the cyclic group generated by the element $b=\pi_{n}(a)$ in the multiplicative group $\IZ_{p^n}^\times$. Since $|\IZ_{p^n}^\times|=p^{n-1}(p-1)$, and $b$ has order $\ge \max\{p,3\}$, the cardinality of the group $B$ is equal to $p^kd$ for some $k\in[1,n-1]$ and some divisor $d$ of the number $p-1$. Moreover, if $p=2$, then $2^k\ge3$ and hence $k\ge 2$ and $n\ge 3$. 

For any number $m\ge n$, consider the quotient homomorphism $\pi_{m,n}:\IZ_{p^m}\to\IZ_{p^n}$, $\pi_{m,n}:x+p^m\IZ\mapsto x+p^n\IZ$. We claim that the subgroup $H=\pi_{m,n}^{-1}(B)$ of the multiplicative group $\IZ_{p^m}^\times$ is cyclic. For odd $p$ this follows from the cyclicity of the group $\IZ_{p^n}^\times$, see Lemma~\ref{l:Gauss}.

For $p=2$, by  Lemma~\ref{l:Gauss}, the multiplicative group $\IZ_{2^m}^\times$ is isomorphic to the additive group $\IZ_2\times\IZ_{2^{m-2}}$. Assuming that $H$ is not cyclic, we conclude that $H$ contains the 4-element Boolean subgroup $$V=\{1+2^m\IZ,-1+2^m\IZ,1+2^{m-1}+2^m\IZ,-1+2^{m-1}+2^m\IZ\}$$ of $\IZ_{2^m}^\times$. Then $B=\pi_{m,n}(H)\supset \pi_{m,n}(V)\ni -1+2^n\IZ$. Taking into account that  $-1+2^n\IZ$ has order $2$ in the cyclic group $B$, we conclude that $-1+2^n\IZ=a^{2^{k-1}}+2^n\IZ$. Since $k\ge 2$, the odd number $c=a^{2^{k-2}}$ is well-defined and $c^2+4\IZ=a^{2^{k-1}}+4\IZ=-1+4\IZ$, which is not possible (as squares of odd numbers are equal  $1$ modulo $4$). This contradiction shows that the group $H$ is cyclic.

By \cite[1.5.5]{Rob}, the number of generators of the cyclic group $H$ can be calculated using the Euler totient function as
\begin{multline*}
\phi(|H|)=\phi(p^{m-n}|B|)=\phi(p^{m-n}p^kd)=\phi(p^{m-n+k})\phi(d)=p^{m-n+k-1}(p-1)\phi(d)=\\
=p^{m-n}\phi(p^k)\phi(d)=p^{n-k}\phi(p^kd)=p^{n-k}\phi(|B|),
\end{multline*}
which implies that for every generator $g$ of the group $B$, every element of the set $\pi_{m,n}^{-1}(g)$ is a generator of the group $H$. In particular, the element $\pi_m(a)\in\pi_{m,n}^{-1}(\pi_n(a))$ is a generator of the group $H$. By the definition of $p$-adic topology,
$$\overline{a^\IN}=\bigcap_{m\ge n}\pi_m^{-1}(\pi_m(a)^\IN)=\bigcap_{m\ge n}\pi_m^{-1}(\pi_{m,n}^{-1}(B))=\bigcap_{m\ge n}\pi_n^{-1}(B)=\pi_n^{-1}(B)=\pi_n^{-1}(\pi_n(a)^\IN).$$
\end{proof}

\begin{lemma}\label{l:S6}  
\begin{enumerate}
\item For any $X\in\XX_p$ there exists  $n\in\IN$ and a cyclic subgroup $H$ of the multiplicative group $\IZ^\times_{p^n}$ such that $X=\pi_n^{-1}(H)$.
\item For any $n\in\IN$ and cyclic subgroup $H$ of $\IZ_{p^n}^\times$ of order $|H|\ge \max\{p,3\}$, there exists a number $a\in \IN\setminus p\IN$ such that $\pi_n^{-1}(H)=\overline{a^\IN}\in\XX_p$.
\end{enumerate} 
\end{lemma}

\begin{proof} 
1. Given any $X\in\XX_p$, find a number $a\in\IN\setminus (\{1\}\cup p\IN)$ such that $X=\overline{a^\IN}$. Choose any $n\in\IN$ with $p^n>a^p$ and observe that the cyclic subgroup $H\subset\IZ_{p^n}^\times$, generated by the element $\pi_n(a)=a+p^n\IZ$, has order $|H|\ge p+1\ge\max\{p,3\}$. By Lemma~\ref{l:S5}, $X=\overline{a^\IN}=\pi_n^{-1}(H)$.
\smallskip

2. Fix $n\in\IN$ and a cyclic subgroup $H$ of $\IZ_{p^n}^\times$ of order $|H|\ge \max\{p,3\}$. Find a number $a\in\IN$ such that the residue class $\pi_n(a)=a+p^n\IZ$ is a generator of the cyclic group $H$. Then $\pi_n(a)$ has order $|H|\ge\max\{p,3\}$, Lemma~\ref{l:S5} ensures that $\pi_n^{-1}(H)=\pi_n^{-1}(\pi_n(a)^\IN)=\overline{a^\IN}\in\XX_p$.
\end{proof}

%Lemmas~\ref{l:S5} and \ref{l:S6} show that the family $\XX_p$ coincides with the family$$\bigcup_{n\in\IN}\{\pi_n^{-1}(H):H\subset \IZ_{p^n}^\times\mbox{ is a cyclic subgroup}\}.$$

For any $X\in\XX_p$, let $$n(X)=\min\big\{n\in\IN:X=\pi_n^{-1}(\pi_n(X)),\;|\pi_n(X)|\ge\max\{p,3\}\big\}.$$
Lemma~\ref{l:S6} implies that the number $n(X)$ is well-defined and $\pi_{n(X)}(X)$ is a cyclic subgroup of order $\ge\max\{p,3\}$ in the multiplicative group $\IZ_{p^{n(X)}}^\times$. Let $i(X)$ be the  index of the subgroup $\pi_{n(X)}(X)$ in $\IZ_{p^{n(X)}}^\times$.

\begin{lemma}\label{l:iso} For any odd prime number $p$ and two sets $X,Y\in\XX_p$ the inclusion $X\subseteq Y$ holds if and only if $i(Y)$ divides $i(X)$.
\end{lemma} 

\begin{proof} Let $m=\max\{n(X),n(Y)\}$. Then $X=\pi^{-1}_m(\pi_m(X))$, $Y=\pi^{-1}_m(\pi_m(Y))$ and $\pi_m(X),\pi_m(Y)$ are subgroups of the multiplicative group $\IZ_{p^m}^\times$, which is cyclic by Gauss Lemma~\ref{l:Gauss}. It follows that the subgroups $\pi_m(X)$ and $\pi_m(Y)$ have indexes $i(X)$ and $i(Y)$ in $\IZ_{p^m}^\times$, respectively. Let $g$ be a generator of the cyclic group $\IZ_{p^m}^\times$. It follows that the subgroups $\pi_m(X)$ and $\pi_m(Y)$ are generated by the elements $g^{i(X)}$ and $g^{i(Y)}$, respectively. Now we see that $X\subseteq Y$ iff $\pi_m(X)\subseteq \pi_m(Y)$ iff $g^{i(X)}\in (g^{i(Y)})^\IN$ iff $i(Y)$ divides $i(X)$.
\end{proof}

\begin{lemma}\label{l:index} For any odd prime number $p$, any $n\in\IN$, and the number $a=1+p^n$ we have $\overline{a^\IN}=1+p^n\IN_0$ and $i(\overline{a^\IN})=p^{n-1}(p-1)$.
\end{lemma}

\begin{proof} Observe that for any $k<p$ we have $a^{k}=(1+p^n)^{k}\in 1+kp^n+p^{n+1}\IZ\ne 1+p^{n+1}\IZ$ and $a^p=(1+p^n)^p\in 1\in p^{n+1}\IZ$, which means that the element $\pi_{n+1}(a)$ has order $p$ in the group $\IZ_{p^{n+1}}^\times$. By Lemma~\ref{l:S5}, $$\overline{a^\IN}=\pi_{n+1}^{-1}(\{a^k+p^{n+1}\IZ:0\le k<p\})=\bigcup_{k=0}^{p-1}(a^k+p^{n+1}\IN_0)=1+p^n\IN_0.$$ 
Also $i(\overline{a^\IN})={|\IZ_{p^{n+1}}^\times|}/{p}=p^{n-1}(p-1)$.
\end{proof}

Lemmas~\ref{l:S6} and \ref{l:iso} imply that for an odd $p$, the poset $\XX_p$ is order isomorphic to the set $$\mathcal D_p=\{d\in\IN:\mbox{$d$ divides $p^n(p-1)$ for some $n\in\IN$}\},$$ endowed with the divisibility relation.

An element $t$ of a partially ordered set $(X,\le)$ is called {\em ${\uparrow}$-chain} if its upper set ${\uparrow}t=\{x\in X:x\ge t\}$ is a chain. 
It is easy to see that the set of ${\uparrow}$-chain elements of the poset $\mathcal D_p$ coincides with the set $\{p^n(p-1):n\in\IN_0\}$ and hence is a well-ordered chain with the smallest element $(p-1)$. 

Below on the Hasse diagrams of the posets $\mathcal D_3$ and $\mathcal D_5$ (showing that these posets are not order isomorphic) the ${\uparrow}$-chain elements are drawn with the bold font.

$$
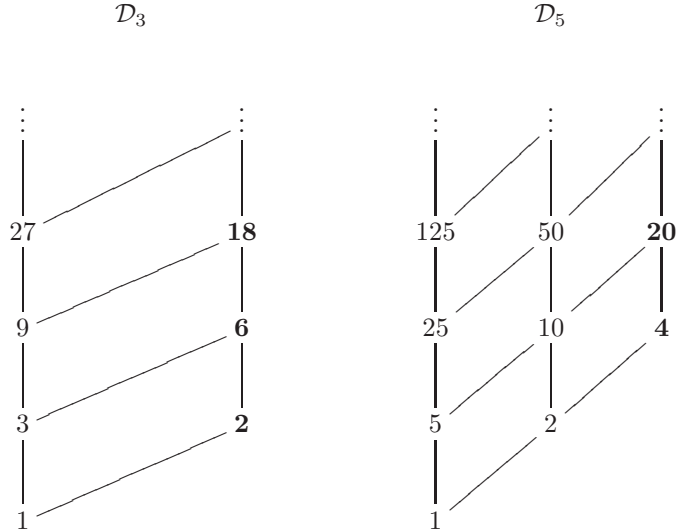
\xymatrix{
&\mathcal D_3&&&&\mathcal D_5\\
\vdots&& \vdots&&\vdots&\vdots&\vdots\\
27\ar@{-}[u]\ar@{-}[urr]&&\mathbf{18}\ar@{-}[u]&&125\ar@{-}[u]\ar@{-}[ur]&50\ar@{-}[u]\ar@{-}[ur]&\mathbf{20}\ar@{-}[u]\\
9\ar@{-}[u]\ar@{-}[urr]&&\mathbf 6\ar@{-}[u]&&25\ar@{-}[u]\ar@{-}[ur]&10\ar@{-}[u]\ar@{-}[ur]&\mathbf 4\ar@{-}[u]\\
3\ar@{-}[u]\ar@{-}[urr]&&\mathbf 2\ar@{-}[u]&&5\ar@{-}[u]\ar@{-}[ur]&2\ar@{-}[u]\ar@{-}[ru]&\\
1\ar@{-}[u]\ar@{-}[urr]&&&&1\ar@{-}[u]\ar@{-}[ur]\\
}
$$

\vskip10pt

Lemmas~\ref{l:iso}, \ref{l:index} and the isomorphness of the posets $\XX_p$ and $\mathcal D_p$ imply the following lemma.

\begin{lemma}\label{l:tree} For an odd prime number $p$, the family $\{1+p^n\IN_0:n\in\IN\}$ coincides with the linearly ordered set of ${\uparrow}$-chain  elements of the poset $\XX_p$.
\end{lemma}

Now we reveal the order structure of the poset $\XX_2$. This poset consists of the closures $\overline{a^\IN}$ in the $2$-adic topology of the sets $a^\IN$ for non-zero odd numbers $a>1$.

\begin{lemma}\label{l:X2}
Let $a\in\IN$ and $X=\overline{a^\IN}$.
\begin{enumerate}
\item If $a\in 1+4\IN$, then $\overline{a^\IN}=1+2^{n(X)-2}\IN_0$.
\item If $a\in 3+4\IN$, then $\overline{a^\IN}=(1+2^{n(X)-1}\IN_0)\cup(-1+2^{n(X)-2}+2^{n(X)-1}\IN_0)$.
\end{enumerate}
In both cases, $i(X)=2^{n(X)-3}$.
\end{lemma}
\begin{proof}
Lemma~\ref{l:S5} and the definition of the number $n(X)$ imply that the projection $C_X:=\pi_{n(X)}(X)=\pi_{n(X)}(a^\IN)$ is a cyclic subgroup of order $4$ of the group $\IZ^\times_{2^{n(X)}}$, and that $X=\pi_{n(X)}^{-1}(C_X)$. In particular, $i(X)=|\IZ_{2^{n(X)}}^\times|/4=2^{n(X)-3}$. 

By the Gauss Lemma~\ref{l:Gauss}, $M_X=\{1+4k+2^{n(X)}\IZ:0\le k<2^{n(X)-2}\}$ is a maximal cyclic subgroup of $\IZ_{2^{n(X)}}$. If $a\in 1+4\IN$, the subgroup generated by $\pi_{n(X)}(a)$ is contained in $M_X$. Then $C_X=\{1+k\cdot 2^{n(X)-2}+2^{n(X)}\IZ\mid 0\leq k<4\}$ and $X=\pi_{n(X)}^{-1}(C_X)=1+2^{n(X)-2}\IZ$.

If $a\in 3+4\IN$, then $C_X$ is not contained in $M_X$. By Gauss Lemma~\ref{l:Gauss} again, the unique cyclic subgroup of $\IZ_{2^{n(X)}}$ of order $4$ not contained in $M_X$ is generated by an element $g$ of $\IZ_{2^{n(X)}}$ such that $-g$ generates the cyclic subgroup of $M_X$ of order $4$. Therefore,
\begin{equation*}
C_X=\{1+2^{n(X)}\IZ,1+2^{n(X)-1}+2^{n(X)}\IZ,-1+2^{n(X)-2}+2^{n(X)}\IZ,-1+2^{n(X)-2}+2^{n(X)-1}+2^{n(X)}\IZ\}.
\end{equation*}
The first two elements, lifted to $\IZ$, give the sequence $1+2^{n(X)-1}\IN$, while the last two give $-1+2^{n(X)-2}+2^{n(X)-1}\IN$. Hence, $X$ is their union.
\end{proof}

\begin{lemma}\label{l:ind2} For every $n\ge 2$ 
\begin{enumerate}
\item the set $X=\overline{(1+2^n)^\IN}\in\XX_2$ coincides with $1+2^n\IN_0$ and has $i(X)=2^{n-1}$;
\item the set $Y=\overline{(-1+2^n)^\IN}\in\XX_2$ coincides with  $(1+2^{n+1}\IN_0)\cup(2^n-1+2^{n+1}\IN_0)$ and has $i(Y)=2^{n-1}$. 
\end{enumerate}
\end{lemma}

\begin{proof} 1. Observe that for every $k<4$ we have $(1+2^n)^k\in 1+k2^n+2^{n+2}\IZ\ne 1+2^{n+2}\IZ$ and $(1+2^n)^4\in 1+2^{n+2}\IZ$, which means that the element $(1+2^n)+2^{n+2}\IZ$ has order 4 in the group $\IZ_{2^{n+2}}^\times$. Then the element $X=\overline{(1+2^n)^\IN}\in\XX_2$ has $n(X)=n+2$ and hence  $X=1+2^n\IN_0$ and $i(X)=2^{n(X)-3}=2^{n-1}$ by Lemma~\ref{l:X2}.
\smallskip

2. Also for every $k<4$ we have $(-1+2^n)^k\in(-1)^k+k2^n+2^{n+2}\IZ\ne 1+2^{n+2}\IZ$  and $(-1+2^n)^4\in 1+2^{n+2}\IZ$, which means that the element $(-1+2^n)+2^{n+2}\IZ$ has order 4 in the group $\IZ_{2^{n+2}}^\times$.  Then the element $Y=\overline{(-1+2^n)^\IN}\in\XX_2$ has $n(Y)=n+2$ and hence  $Y=(1+2^{n+1}\IN_0)\cup(2^n-1+2^{n+1}\IN_0)$ and $i(Y)=2^{n(Y)-3}=2^{n-1}$ by Lemma~\ref{l:X2}.
\end{proof}

\begin{lemma}\label{l:ord2} For distinct sets $X,Y\in\XX_2$, the inclusion $X\subset Y$ holds if and only if  $X\subseteq 1+4\IN_0$ and $i(Y)<i(X)$.
\end{lemma} 
\begin{proof}
If $X\subseteq 1+4\IN_0$, then by Lemma \ref{l:X2}, $X=1+2^{n(X)-1}\IN_0$. If $i(Y)<i(X)$, then $n(Y)<n(X)$, and thus by Lemma \ref{l:X2} we get $Y\supset 1+2^{n(X)-1}\IN_0$ (i.e., $Y\supset X$) both if $Y$ is contained in $1+4\IN_0$ or if it is not.

Conversely, if $X\subset Y$, the claim follows by writing explicitly $X$ and $Y$ through Lemma \ref{l:X2}.
\end{proof}

Lemmas~\ref{l:ind2} and \ref{l:ord2} imply:

\begin{lemma}\label{l:anti2} The family $\min\XX_2=\{X\in\XX_2:X\not\subseteq 1+8\IN_0\}$ coincides with the set of minimal elements of the poset $\XX_2$ and the set $\XX_2\setminus \min\XX_2=\{X\in\XX_2:X\subseteq 1+8\IN_0\}$ is linearly ordered and coincides with the set $\{1+2^n\IN_0:n\ge 3\}$.
\end{lemma}

$$
\begin{gathered}
\xymatrix{
\vdots\ar@{-}[dr]&\vdots\\
\overline{33^\IN}\ar@{-}[u]\ar@{-}[dr]&\overline{31^\IN}\\
\overline{17^\IN}\ar@{-}[u]\ar@{-}[dr]&\overline{15^\IN}\\
\overline{9^\IN}\ar@{-}[u]\ar@{-}[dr]&\overline{7^\IN}\\
\overline{5^\IN}\ar@{-}[u]&\overline{3^\IN}\\
}\\
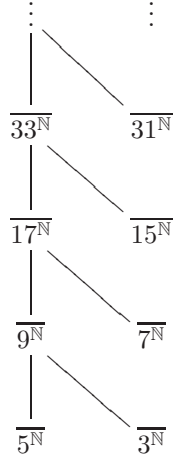
\mbox{The Hasse diagram of the poset $\XX_2$}
\end{gathered}
$$
\vskip10pt

\begin{lemma}\label{l:3} For any homeomorphism $h$ of the Golomb space $\IN_\tau$ and any $n\in\{1,2,3\}$ we have $h(n)=n$.
\end{lemma}

\begin{proof} 1. The equality $h(1)=1$ follows from Lemma~\ref{l:BMT}(1). 
\smallskip

2. By Lemma~\ref{l:isomor}, $h$ induces an order isomorphism of the posets $\XX_2$ and $\XX_{h(2)}$. By Lemma~\ref{l:ind2}, the set $\{\overline{(-1+2^n)^\IN}:n\ge 2\}$ is an infinite antichain in the poset $\XX_2$. Consequently, the poset $\XX_{h(2)}$ also contains an infinite antichain. 
On the other hand, for any odd prime number $p$ the poset $\XX_p$ is order-isomorphic to the poset $\mathcal D_p$, which contain no infinite antichains. Consequently, $\XX_{h(2)}$ cannot be order isomorphic to $\XX_p$, and hence $h(2)=2$.
\smallskip

3. By Lemma~\ref{l:BMT}(2), $h(3)$ is a prime number, not equal to $h(2)=2$. By Lemma~\ref{l:isomor}, $h$ induces an order isomorphism of the posets $\XX_3$ and $\XX_{h(3)}$.  Then the posets $\mathcal D_3$ and $\mathcal D_{h(3)}$ also are order isomorphic. The smallest ${\uparrow}$-chain element of the poset $\mathcal D_3$ is $2$ and the set ${\downarrow}2=\{d\in\mathcal D_3:d$ divides $2\}$ has cardinality $2$. On the other hand, the smallest ${\uparrow}$-chain element of the poset $\mathcal D_{h(3)}$ is $h(3)-1$. Since the sets $\mathcal D_3$ and $\mathcal D_{h(3)}$ are order-isomorphic, the set ${\downarrow}(h(3)-1)=\{d\in\mathcal D_p:d$ divides $h(3)-1\}$ has cardinality $2$, which means that the number $h(3)-1$ is prime. Observing that $3$ is a unique odd prime number $p$ such that $p-1$ is prime, we conclude that $h(3)=3$.
\end{proof}

\begin{lemma}\label{l:main-p} For any homeomorphism $h$ of the Golomb space $\IN_\tau$, and any prime number $p$ we have $h(1+p^n\IN_0)=1+h(p)^n\IN_0$ for all $n\in\IN$.
\end{lemma}

\begin{proof} By Lemma~\ref{l:isomor}, the homeomorphism $h$ induces an order isomorphism of the posets $\XX_p$ and $\XX_{h(p)}$.

If $p=2$, then $h(p)=2$ by Lemma~\ref{l:3}. Consequently, $h$ induces an order automorphism of the poset $\XX_2$ and hence $h$ is identity on the well-odered set $\{1+2^n\IN_0:n\ge 3\}$ of non-minimal elements of $\XX_2$. Consequently, $h(1+2^n\IN_0)=1+2^n\IN_0$ for all $n\ge 3$. 

Next, we show that $h(1+4\IN_0)=1+4\IN_0$. Observe that for the smallest non-minimal element $\overline{9^\IN}=1+8\IN_0$ of $\XX_2$ there are only two elements $\overline{5^\IN}=1+4\IN_0$ and $\overline{3^\IN}=(1+8\IN_0)\cup(3+8\IN_0)$, which are strictly smaller than $\overline{9^\IN}$ in the poset $\XX_2$. Then $h(\overline{5^\IN})\in\{\overline{3^\IN},\overline{5^\IN}\}$. By Lemma~\ref{l:3}, $h(3)=3$ and hence $h(\overline{3^\IN})=\overline{3^\IN}$, which implies that $h(1+4\IN_0)=h(\overline{5^\IN})=\overline{5^\IN}=1+4\IN_0$.

Now assume that $p$ is an odd prime number. Since $h(2)=2$, the prime number $h(p)\ne h(2)=2$ is odd. 
 By Lemma~\ref{l:tree}, the well-ordered sets $\{1+p^n\IN_0:n\in\IN\}$ and $\{1+h(p)^n\IN_0:n\in\IN\}$ coincide with the sets of ${\uparrow}$-chain elements of the posets $\XX_p$ and $\XX_{h(p)}$, respectively. Taking into account that $h$ is an order isomorphism, we conclude that $h(1+p^n\IN_0)=1+h(p)^n\IN_0$ for every $n\in\IN$. 
\end{proof}

\section{Proof of Theorem~\ref{t:main}}\label{s:main}

In this section we present the proof of Theorem~\ref{t:main}. Given any homeomorphism $h$ of the Golomb space $\IN_\tau$, we need to prove that $h(n)=n$ for all $n\in\IN$. This equality will be proved by induction.

For $n\le 3$ the equality $h(n)=n$ is proved in Lemma~\ref{l:3}. Assume that for some number $n\ge 4$ we have proved that $h(k)=k$ for all $k<n$. For every prime number $p$ let $\alpha_p $ be the largest integer number such that $p^{\alpha_p}$ divides $n-1$ (so, $\alpha_p=l_p(n-1)$). For every $p\in\Pi_{n-1}$ we have $p\le n-1$ and hence $h(p)=p$ (by the inductive hypothesis). Then $h(\Pi_{n-1})=\Pi_{n-1}$ and $h(\Pi\setminus \Pi_{n-1})=\Pi\setminus\Pi_{n-1}$.

Observe that $n$ is the unique element of the set $$\bigcap_{p\in\Pi}(1+p^{\alpha_p}\IN_0)\setminus(1+p^{\alpha_p+1}\IN_0).$$
By Lemma~\ref{l:main-p}, $h(n)$ coincides with the unique element of the set
$$
\begin{aligned}
&\bigcap_{p\in\Pi}(1+h(p)^{\alpha_p}\IN_0)\setminus(1+h(p)^{\alpha_p+1}\IN_0)=\\
&\Big(\bigcap_{p\in\Pi_{n-1}}\!\!(1{+}h(p)^{\alpha_p}\IN_0)\setminus(1{+}h(p)^{\alpha_p{+}1}\IN_0)\Big)\cap\Big(\bigcap_{p\in\Pi\setminus \Pi_{n-1}}\!\!\IN\setminus(1{+}h(p)\IN_0)\Big) =\\
&\Big(\bigcap_{p\in\Pi_{n-1}}(1+p^{\alpha_p}\IN_0)\setminus(1+p^{\alpha_p+1}\IN_0)\Big)\cap\Big(\bigcap_{p\in\Pi\setminus \Pi_{n-1}}\IN\setminus(1+p\IN_0)\Big)=\{n\}
\end{aligned}
$$and hence $h(n)=n$.
%\newpage

\end{document}